\documentclass[12pt]{article}
\usepackage[cp1251]{inputenc}
\usepackage{amsmath,amsthm,amsfonts}

\date{}
\newtheorem{Teo}{Theorem}

\newtheorem{car}{Corollary}

\newtheorem{defin}[Teo]{Definition}

\begin{document}

\title{{\large \bf Some remarks about disjointly homogeneous symmetric spaces}}
\author {Sergey V. Astashkin}

\date{}

\maketitle

Department of Mathematics, Samara National Research University, Moskovskoye shosse 34, 443086, Samara, Russia

\renewcommand{\thefootnote}{\fnsymbol{footnote}}
\footnotetext{$^*$The work was supported by the Ministry of Education and Science of the Russian Federation, 
project 1.470.2016/1.4 and by the RFBR grant 18--01--00414.}



\maketitle

\def\mn{||w_m||_2}
\def\no{\sum_{k=n_m}^{n_{m+1}-1}}
\def\op{\frac{dt}{{\psi}(t)}}
\def\pq{\log_2^{1/2}{(4/t)}}
\def\qr{\max_{0\leq m\leq r}}
\def\rs{\max_{n_m\leq k<n_{m+1}}}
\def\st{\biggl\|\sum_{m=0}^{r}\,c_mv_m\biggr\|}
\def\tz{\int_0^1\,\frac{d\psi(s)}{\varphi(s)}}
\def\zx{\sum_{k=0}^{\infty}\,\frac{\psi(2^{-k})-\psi(2^{-k-1})}{\varphi(2^{-k})}}
\def\xy{\sum_{k=0}^{\infty}\,\frac{a_k}{\sqrt{S_k}\varphi(2^{-k})}}
\def\ya{M(\tilde{\rho})\subset{\Lambda(\psi)}}


\begin{abstract}
\noindent Let $1\le p<\infty$. A symmetric space $X$ on $[0,1]$ is said to be $p$-disjointly homogeneous (resp.  restricted $p$-disjointly homogeneous) if every sequence of normalized pairwise disjoint functions from $X$ (resp. characteristic functions) contains a subsequence equivalent in $X$ to the unit vector basis of $l_p$. Answering a question posed in the paper \cite{HST}, we construct, for each $1\le p<\infty$, a restricted $p$-disjointly homogeneous symmetric space, which is not $p$-disjointly homogeneous. Moreover, we prove that  the property of $p$-disjoint homogeneity is preserved under Banach isomorphisms.
\end{abstract}

\maketitle

\vspace{5 mm}

\section{Introduction}\label{Intro}

\noindent


A Banach lattice $E$ is called {\it disjointly homogeneous} (shortly DH) if two arbitrary sequences of normalized pairwise disjoint elements in $E$ contain equivalent subsequences. In particular, given $1\le p\le\infty$, a Banach lattice is  {\it $p$-disjointly homogeneous} (shortly p-DH) if each  normalized disjoint sequence has a subsequence equivalent to the unit vector basis of $l_p$ ($c_0$ when $p=\infty$). These notions were first introduced in \cite{FTT-09} and proved to be very useful in studying the general problem of identifying Banach lattices $E$ such that the ideals of strictly singular and compact operators  in $E$ coincide \cite{bib:04} (see also survey \cite{FHT-survey} and references therein). Results obtained there can be treated as a continuation and development of a classical theorem of V.~D.~Milman \cite{Mil-70} which states that every strictly singular operator in $L_p(\mu)$ has compact square.

Recently, in \cite{HST}, in the setting of symmetric spaces it was introduced a weaker property of restricted disjoint homogeneity. A symmetric space $X$ on $[0,1]$ is said to be {\it restricted $2$-disjointly homogeneous} (shortly restricted 2-DH) if every sequence of normalized disjoint characteristic functions contains a subsequence 
equivalent to the unit vector basis of $l_2$. Clearly, each 2-DH symmetric space is restricted 2-DH. In \cite{HST}, the authors proved the converse for Orlicz spaces \cite[Theorem~5.1]{HST} and also asked whether a symmetric space $X$ on $[0,1]$ which is restricted 2-DH, must be 2-DH. This question (repeated also in \cite[p.~19]{FHT-survey}) was  motivated by the fact that restricted 2-DH symmetric spaces have rather "good"\;properties. In particular, they are stable under duality \cite[Proposition~3.7]{HST} while this is still open problem for 2-DH symmetric spaces, see \cite[p.~20]{FHT-survey}. Moreover, every symmetric space isomorphic (as a Banach space) to a 2-DH symmetric space $Y$ is restricted 2-DH \cite[Corollary~3.6]{HST}. Observe that an analogous result for the 2-DH property was unknown. 

In this paper we solve the problem whether a restricted 2-DH symmetric space on $[0,1]$ is also 2-DH in the negative. More precisely, given $1\le p,q<\infty$ we construct a restricted $p$-DH symmetric space $Z_{p,q}$ on $[0,1]$, which contains a sequence of pairwise disjoint functions equivalent to the unit vector basis of $l_q$. Clearly, if $p\ne q$ the space $Z_{p,q}$ is not DH. We show also that in the case when $1< p,q<\infty$ the space $Z_{p,q}$ is reflexive. Moreover,  $Z_{p,1}$, $1\le p<\infty$, is a disjointly complemented space (i.e., every sequence of pairwise disjoint functions from $Z_{p,1}$ has a subsequence whose span is complemented in $Z_{p,1}$).

At the same time, by using a deep result by Kalton
on uniqueness of rearrangement invariant structures \cite[Theorem~7.4]{K-93} (see also \cite[Theorem~5.7]{Tz}), we prove that in the setting of symmetric spaces on $[0,1]$ the $p$-DH property is preserved under Banach isomorphisms, i.e., if $1\le p<\infty$ then each symmetric space isomorphic to a $p$-DH symmetric space is also $p$-DH.

\section{Preliminaries}\label{Intro2}

In this section, we shall briefly list the
definitions and notions used throughout this paper. For more detailed information, we refer to the monographs \cite{LT-79,KPS,bs}.

A Banach space $(X,\Vert \cdot\Vert _{_{X}})$ of real-valued
Lebesgue measurable functions (with identification $m $-a.e.) on the interval $[0,1]$ is called {\it symmetric (or rearrangement invariant)} if
\begin {enumerate}
\item[(i).]  $X$ is an ideal lattice, that is, if $y\in X$ and $x$ is any measurable function on $[0,1]$ with $\vert x\vert
\leq \vert y\vert $, then $x\in X$ and $\Vert x\Vert _{_{X}} \leq \Vert y\Vert _{_{X}};$
\item[(ii).]  $X$ is symmetric in the sense that if
 $y\in X$, and if $x$
is any measurable function on $[0,1]$ with $x^*=y^*$, then $x\in
X$ and $\Vert x\Vert _{_{X}} = \Vert y\Vert _{_{X}}$.
\end{enumerate}
\bigskip
\noindent Here, $x^*$ denotes the non-increasing, right-continuous rearrangement of a measurable function $x$ on $[0,1]$ given by
$$
x^{*}(t)=\inf \{~s\ge 0:\,m\{u\in [0,1]: |x(u)|>s\}\le t~\},\quad t>0,
$$
where $m $ is the usual Lebesgue measure.

For any symmetric space  $X$ on $[0,1]$ we have $L_\infty [0,1]\subseteq X\subseteq L_1[0,1]$. The fundamental function $\phi_X$ of a symmetric space $X$ is defined by  $\phi_X(t):=\|\chi_{[0,t]}\|_X$. In what follows $\chi_A$ is the characteristic function of a set $A$.

The {\it K\"othe dual} (or the {\it associated} space) $X'$ of a symmetric space $X$  consists of all measurable functions $y$, for which
$$
\Vert y\Vert _{_{X'}}:= \sup \Big\{\int _{0}^1\vert
x(t)y(t)\vert dt:\ x\in X,\ \Vert x \Vert _{_{X}}\leq 1\Big\}
<\infty.
$$
If $X^*$ denotes the Banach dual of $X$, then $X' \subset X^{*}$ and $X'=X^{*}$ if and only if the
norm $\Vert \cdot \Vert _{_{X}}$ is order-continuous, i.e., from
$\{x_{n}\}\subseteq X, x_{n} \downarrow _n 0$, it follows that
$\Vert x_{n}\Vert _{_{X}}\rightarrow 0$. Note that the norm
$\Vert \cdot \Vert _{_{X}}$ of the symmetric space $X$
 is order-continuous if and only if $X$ is separable (see e.g. \cite[Ch.~IV, Theorem~3.3]{KA77}). 
We denote by $X_0$ the closure of $L_\infty$ in $X$ (the {\it separable part} of $X$). The space
$X_0$ is symmetric, and it is separable if $X\ne L_\infty$.


Let us recall some classical examples of symmetric spaces on $[0,1]$. Denote by $\Omega$ the set of all increasing concave functions $\varphi$ on $[0,1]$ such that $\varphi(0)=0$. Every $\varphi\in\Omega$ and $1\le q<\infty$ generate the {\it Lorentz} space $\Lambda_q(\varphi)$ endowed with the norm
\[
\|x\|_{\Lambda_q(\varphi)}:=\Big(\int\limits_0^1 x^*(t)^q\,d\varphi(t)\Big)^{1/q}.\]
We set $\Lambda(\varphi):=\Lambda_1(\varphi)$.

If $\varphi\in\Omega$ then the {\it Marcinkiewicz} space $M(\varphi)$ consists of all measurable functions $x$ such that
\begin{equation*}
 \|x\|_{M(\varphi)}:=\sup\limits_{0<\tau\leq
1}\frac{1}{\varphi(\tau)}\int\limits_0^\tau x^*(t)\,dt<\infty.
\end{equation*}
The space $M(\varphi)$ is not separable provided that $\lim_{\tau\to 0}\varphi(\tau)=0$ (or, equivalently, $M(\varphi)\ne L_1$). At the same time, its subspace $M_0(\varphi)$, consisting of all $x\in M(\varphi)$ such that
$$\lim_{\tau\to 0}\frac{1}{\varphi(\tau)}\int_0^\tau x^*(t)\,dt=0,$$
is a separable symmetric space which, in fact, coincides with the separable part $(M(\varphi))_0$. We have $(\Lambda(\varphi))'=M(\varphi)$ and $(M(\varphi))' = (M_0(\varphi))'= \Lambda(\varphi)$ \cite[Theorems~II.5.2 and II.5.4]{KPS}.

For each $\varphi\in\Omega$ the spaces $\Lambda(\varphi)$ and  $M(\bar\varphi)$, where $\bar\varphi(t):=t/\varphi(t)$, are the smallest and the largest ones in the class of all symmetric spaces with the fundamental function $\varphi$, i.e., $\Lambda(\varphi)\subset X\subset M(\bar\varphi)$
whenever $X$ is a symmetric space such that $\phi_X=\varphi$ \cite[Theorems~II.5.5 and II.5.7]{KPS}.

The behaviour of a function $\varphi\in\Omega$ is essentially determined by the numbers
\[
 \gamma_\varphi:= \lim\limits_{t\to+0}
 \frac{\ln M_\varphi(t)}{\ln t}\;\;\mbox{and}\;\;
 \delta_\varphi:= \lim\limits_{t\to\infty}
 \frac{\ln M_\varphi(t)}{\ln t},
\]
where 
\[
 M_\varphi(t):=\sup\limits_{0<s\leq \min
 \left(1, \frac1t\right)} \frac{\varphi(ts)}{\varphi(s)}.
\]
For each $\varphi\in\Omega$ the inequalities $0\le \gamma_\varphi\le \delta_\varphi\le 1$ hold \cite[\S\,II.1]{KPS}.
In the case when $\gamma_\varphi>0$ we have
\begin{equation*}\label{Prel: Eq2}
||x||_{M(\varphi)}\asymp \sup_{0<t\leq 1}(x^*(t)\bar{\varphi}(t))
\end{equation*}
with constants independent of $x\in M(\varphi)$ (see \cite[Theorem~2.5.3]{KPS}).


Let $1<q<\infty$ and let $X$ be a symmetric space on $[0,1]$. Denote by $X^{(q)}$ the {\it $q$-convexification} of $X$ defined as $X^{(q)}:=\{x\;\mbox{measurable on}\;[0,1]:\;|x|^q\in X\}$ with the norm $\|x\|_{X^{(q)}}:=\|\,|x|^q\,\|_X^{1/q}$ (see \cite[p.~53]{LT-79}). 

Next, we will make use of the {\it real interpolation method} \cite{BL}. For a pair of symmetric spaces $(X_0, X_1)$ the {\it Peetre K-functional} of an element $x \in X_0+X_1$ is defined for $t > 0$ by
$$
K(t, x; X_0, X_1) = \inf \{ \| x_0\|_{X_0} + t \| x_1\|_{X_1}: x = x_0 + x_1, x_0 \in X_0, x_1 \in X_1 \}.
$$
Then, the real Lions-Peetre interpolation spaces are defined as follows
$$
(X_0, X_1)_{\theta, p} = \{ x \in X_0+X_1: \| x \|_{\theta, p} = \Big(\int_0^{\infty} [t^{-\theta} K(t, x; X_0, X_1)]^p \frac{dt}{t}\Big)^{1/p}
< \infty \}
$$
if $0 < \theta < 1$ and $1 \leq p < \infty$, and
$$
(X_0, X_1)_{\theta, \infty} = \{ x \in X_0+X_1: \| x \|_{\theta, \infty} = \sup_{t > 0} \frac {K(t, x; X_0, X_1)}{t^{\theta} } < \infty \}
$$
if $0 \leq \theta \leq 1$.

Convergence in measure (resp. in weak topology) of a
sequence of measurable functions $\{x_n\}_{n=1}^\infty$
(resp. from a symmetric space $X$) to a measurable function $x$
(resp. from $X$) is denoted by $x_n \stackrel{m}{\to} x$
(resp. $x_n \stackrel{w}{\to}x$).
The notation   $A\asymp B$ will mean that there exist constants $C>0$ and $c>0$ not depending on the arguments of $A$ and $B$ such that $c{\cdot}A\le B\le C{\cdot}A$. 
Moreover, throughout the paper $\|f\|_p:=\|f\|_{L_p[0,1]}$, $1\le p\le\infty$. 


\section{Restricted $p$-DH symmetric spaces which are not $p$-DH}

We start with the following definitions. 

\begin{defin}\cite{FTT-09}
A symmetric space $X$ on $[0,1]$ is {\it disjointly homogeneous} (shortly DH) if two arbitrary normalized disjoint sequences from $X$ contain equivalent subsequences.

Given $1\le p\le\infty$, a symmetric space $X$ on $[0,1]$ is called  {\it $p$-disjointly homogeneous} (shortly p-DH) if each  normalized disjoint sequence has a subsequence equivalent in $X$ to the unit vector basis of $l_p$ ($c_0$ when $p=\infty$).
\end{defin}

For examples and other information related to DH and $p$-DH symmetric spaces and Banach lattices see \cite{FTT-09,bib:04,FHSTT,FHT-survey,HST,A15}.


\begin{defin}\cite{HST}
Let $1\le p\le\infty$. A symmetric space $X$ on $[0,1]$ is said to be {\it restricted $p$-DH} if for every sequence of pairwise disjoint subsets $\{A_n\}_{n=1}^\infty$ of $[0,1]$ there is a subsequence $\{A_{n_k}\}$ such that $\{\frac{1}{\|\chi_{A_{n_k}}\|_X} \chi_{A_{n_k}}\}$ is equivalent to the unit vector basis of $l_p$ ($c_0$ when $p=\infty$).
\end{defin}

\begin{defin}\cite{FHSTT}
A symmetric space $X$ on $[0,1]$ is called {\it disjointly complemented} ($X\in DC$) if every disjoint sequence from $X$ has a subsequence whose span is complemented in $X$.
\end{defin}

Clearly, each $p$-DH symmetric space is restricted $p$-DH. In \cite{HST}, there was posed the question whether a symmetric space $X$, which is restricted $p$-DH, is $p$-DH (see also \cite[p.~19]{FHT-survey}). The following theorem solves this problem in the negative.

\begin{Teo}\label{main}
Let $1\le p,q<\infty$. There exists a restricted $p$-DH symmetric space $Z_{p,q}$ on $[0,1]$, which contains a sequence of pairwise disjoint functions $\{g_m\}$ equivalent to the unit vector basis of $l_q$ such that the closed linear span $[g_m]$ is complemented in $Z$. If $1< p,q<\infty$ then the space $Z_{p,q}$ is reflexive. Moreover, $Z_{p,1}\in DC$ for each $1\le p<\infty$.
\end{Teo}

\begin{proof}
To do the structure of the proof more understandable and transparent, split it into three parts.

{\it Step 1.} Following an idea of the proof of Theorem~3 from \cite{A99}, we construct two separable symmetric spaces  $E_0$ and $E_1$ on $[0,1]$ with  fundamental functions $\psi$ and $\varphi$, respectively, such that $E_0\subset E_1$, $E_0$ contains a sequence of pairwise disjoint functions $\{v_m\}\subset E_0$, which is equivalent to the unit vector basis of $c_0$ both in $E_0$ and in $E_1$, and
\begin{equation}\label{Eq1}
\lim_{t\to 0}\frac{\varphi(t)}{\psi(t)}=0.
\end{equation}

Take for $E_0$ the space $M_0(\bar{\psi})$ (i.e., the separable part of the Marcinkiewicz space $M(\bar{\psi})$), where 
$\psi(t)=t^{1/2}\log_2^{1/2}({4}/{t})$, $0<t\leq 1$ (recall that $\bar{\psi}(t)=t/\psi(t)$). 
Thus,
$$
\|x\|_{E_0}= \sup_{0<t\leq 1}\frac{\psi(t)}{t}\int_0^tx^*(s)\,ds\;\;\mbox{and}\;\;\lim_{t\to 0}\frac{\psi(t)}{t}\int_0^tx^*(s)\,ds=0.$$
Then, since $\gamma_{\psi}=1/2$, by \cite[Theorem~2.5.3]{KPS} (see also Section~\ref{Intro2}), we have
\begin{equation}\label{Eq2}
||x||_{E_0}\asymp \sup_{0<t\leq 1}(x^*(t){\psi}(t)),
\end{equation}
and from the inequality $x^*(t)\le \frac1t\int_0^t x^*(s)\,ds$, $0<t\le 1$, for every $x\in E_0$ it follows 
\begin{equation}\label{Eq2a}
\lim_{t\to 0}(x^*(t){\psi}(t))=0.
\end{equation}

Let us define the space $E_1$. We put $\alpha_k:=1/\psi(2^{-k})=(k+2)^{-1/2}2^{k/2}$ and $z_k(t):=\alpha_k\chi_{(0,2^{-k}]}(t)$, $k=0,1,\dots$ Moreover, we define the sequence of positive integers $
\{n_m\}_{m=0}^\infty$ such that $n_0=1<n_1<n_2<\dots <n_m<\dots$ by setting
\begin{equation}\label{Eq3}
n_{m}:=\max\Big\{n=1,2,\dots:\,\sum_{k=n_{m-1}}^{n-1}\frac{1}{k+2}\leq 1\Big\},\;\;m=1,2,\dots
\end{equation}
Then, we denote 
$$
w_m(t):=\,\max_{n_m\leq k<n_{m+1}}z_k(t),\;\;m=0,1,\dots$$
Since the sequence $\{\alpha_k\}_{k=0}^\infty$ increases, from \eqref{Eq3} it follows that the $L_2$-norms of the functions $w_m$ satisfy the estimates
$$
||w_m||_2^2\,\geq{\,\sum_{k=n_m}^{n_{m+1}-1}\alpha_k^22^{-k-1}}\,=\,
\frac12\no\frac{1}{k+2}\,\geq\,\frac{1}{4}$$
and
$$
{\mn}^2\,\leq{\,{\no}\alpha_k^22^{-k}}\,=\,\no\frac{1}{k+2}\,\leq\,1,$$
whence
\begin{equation}\label{Eq4}
\frac{1}{2}\leq\mn\leq 1,\;\;m=0,1,\dots
\end{equation}

Further, we denote $\chi_b:=b^{-1/2}\chi_{(0,b)},$ $0<b\le 1$, $\bar{w}_m:=w_m/{||w_m||_2},$ $m=0,1,\dots$, and define the set $V$ as follows
$$
V:=\,\{\chi_b\}_{0<b\leq 1}\bigcup{\{\bar{w}_m\}_{m=0}^{\infty}}.$$
Moreover, let $E_1$ consist of all measurable functions $x(t)$ on
$[0,1]$ such that
\begin{equation}\label{Eq4a}
\lim_{s\to +0}\sup_{v\in V}\int_0^sx^*(t)v(t)\,dt=0.
\end{equation}
Let us show that $E_1$ endowed with the norm
$$
||x||_{E_1}:=\sup_{v\in V}\int_0^1x^*(t)v(t)\,dt$$
is a separable symmetric space on $[0,1]$. As was observed in Section~\ref{Intro2} (see also \cite[Ch.~IV, Theorem~3.3]{KA77}) for this it is sufficient to check that  the norm of $E_1$ is order-continuous.

Suppose that $\{x_{n}\}\subseteq E_1$, $x_{n} \downarrow _n 0$. Then, for each $s>0$ and all sufficiently large $n\in\mathbb{N}$ we have
$$
m(\{t:\,x_n(t)\ge s\})\le s.$$
Therefore, since $\|v\|_2\le 1$ for every $v\in V$, we have
\begin{eqnarray*}
\|x_n\|_{E_1} &=& \sup_{v\in V}\int_0^1x_n^*(t)v(t)\,dt \le \sup_{v\in V}\int_0^sx_n^*(t)v(t)\,dt +s\sup_{v\in V}\int_s^1v(t)\,dt\\
&\le& \sup_{v\in V}\int_0^sx_1^*(t)v(t)\,dt +s
\end{eqnarray*}
whenever $n$ is sufficiently large. Combining this inequality with the fact that $x_1\in E_1$, we see that the right-hand side of the latter inequality tends to zero as $s\to 0$. Consequently,
$$
\lim_{n\to\infty}\|x_n\|_{E_1}=0,$$
and so the norm of $E_1$ is order-continuous.

In addition, from the definition of the norm of $E_1$ it follows that 
$$
||x||_{M(t^{1/2})}\leq{||x||_{E_1}}\leq{||x||_2}.$$
Therefore, $\varphi(t)=t^{1/2}$ is the fundamental function of $E_1$. Thereby the fundamental functions of the spaces $E_0$ and $E_1$ (i.e., the functions $\psi$ and $\varphi$) satisfy condition \eqref{Eq1}.

Let us show that
\begin{equation}\label{Eq5}
E_0\subset E_1.
\end{equation}
Firstly, we check that 
\begin{equation}\label{Eq5a}
\sup_{v\in V}\int_0^1\frac{v(t)}{{\psi}(t)}\,dt<\infty.
\end{equation}
Indeed,
$$
\int_0^1\chi_b(t)\op\,=\,2b^{-1/2}\int_0^b\frac{d(t^{1/2})}{\pq}\leq 2\;\;\mbox{for all}\;\;0<b\leq 1,$$
and, by \eqref{Eq3},
\begin{eqnarray*}
\int_0^1w_m(t)\op &\leq& {\no}\alpha_k\int_0^{2^{-k}}\op=2{\no}\alpha_k\int_0^{2^{-k}}\frac{d(t^{1/2})}{\pq}\\ &\leq&
2{\no}\frac{1}{k+2}\leq\,2.
\end{eqnarray*}
Combining the last inequalities with \eqref{Eq4} yields \eqref{Eq5a}. Now, let $x\in E_0$ be arbitrary. Then, we have
$$
\sup_{v\in V}\int_0^sx^*(t)v(t)\,dt\le \sup_{0<t\le s}(x^*(t){\psi}(t))\cdot \sup_{v\in V}\int_0^1\frac{v(t)}{{\psi}(t)}\,dt.$$
Hence, from \eqref{Eq2a} and \eqref{Eq5a} it follows \eqref{Eq4a},
i.e., $x\in E_1$.
Thus, embedding
\eqref{Eq5} is proved.

Next, we set $D_m:=(2^{-n_{m+1}},2^{-n_m}]$ and
$$
v_m(t):=w_m(t)\chi_{D_m}(t)={\no}\alpha_k\chi_{(2^{-k-1},2^{-k}]}(t),\;\;m=0,1,\dots$$
Clearly, the functions $v_m$, $m=0,1,\dots$, are pairwise disjoint. We  show that the sequence $\{v_m\}$ is equivalent to the unit vector basis of $c_0$ both in $E_0$ and in $E_1$.  

Let  
$$
v(t)=\sum_{m=0}^rc_mv_m(t),\;\;0<t\le 1,$$
where $r\in\mathbb{N}$, $c_m\in\mathbb{R}$, $m=0,1,\dots,r$. Without loss of generality, we can assume that $c_m\geq 0.$
Then, $w(t):={\qr}c_mw_m(t)$ is a non-increasing function on $(0,1]$
and $v(t)\leq w(t).$ Therefore, from \eqref{Eq2} it follows that
$$
||v||_{E_0}\,\leq\,||w||_{E_0}\,\leq{\,C\qr\left\{c_m{\rs}\alpha_k{\psi}(2^{-k})\right\}}.$$
Hence, in view of the fact that $\alpha_k{\psi}(2^{-k})=1$ for $k=0,1,2,\dots,$ we obtain 
\begin{equation}\label{Eq6}
||v||_{E_0}\,\leq{\,C{\qr}c_m}.
\end{equation}

Now, let us estimate the norm $||v||_{E_1}$ from below. By \eqref{Eq3}, for each $m=0,1,\dots, r$ we infer
\begin{eqnarray*}
\int_0^1v_m^*(t)w_m(t)\,dt &\geq& \int_{D_m}w_m^2(t)\,dt=\int_0^1v_m^2(t)\,dt\\ &=&{\no}\alpha_k^22^{-k-1}=\frac{1}{2}\no\frac{1}{k+2}\geq\frac{1}{4}.
\end{eqnarray*}
Combining this together with \eqref{Eq4} and with the definition of the norm in $E_1$, we obtain $||v_m||_{E_1}\geq{1/4}.$  Therefore, 
$$
||v||_{E_1}\,\geq{\,\qr\{c_m||v_m||_{E_1}\}}\,\geq{\,\frac{1}{4}{\qr}c_m},$$
and so, according to \eqref{Eq5} and \eqref{Eq6}, there exists a constant $B>0$ such that, for arbitrary $r\in \mathbb{N}$ and all $c_m\in \mathbb{R}$, we have
\begin{equation}\label{Eq7}
B^{-1}\qr |c_m|\leq{\st}_{E_0}\leq B\qr |c_m|
\end{equation}
and
$$
B^{-1}\qr |c_m|\leq{{\st}_{E_1}}\leq B\qr |c_m|.$$
This completes Step 1.

{\it Step 2.} We apply a simple duality argument. Since the spaces $E_0$ and $E_1$ are separable, the (Banach) dual spaces $E_0^*$ and $E_1^*$ coincide (isometrically) with their K\"{o}the duals $E_0'=(M_0(\bar{\psi}))'=\Lambda(\bar{\psi})$ and $E_1'$, which have the fundamental functions $t^{1/2}\log_2^{-1/2}\frac{4}{t}$ and $t^{1/2}$, respectively \cite[\S\,II.4]{KPS}. Clearly, $E_1'\subset E_0'$. 
 
Let $\{u_m\}\subset E_1'$ be a sequence of pairwise disjoint functions such that $\|u_m\|_{E_1'}\asymp 1$, $m=0,1,2,\dots$, $\int_0^1 v_mu_m\,dt=1$ and  $\int_0^1 v_mu_n\,dt=0$ if $m\ne n$. Let us show that $\{u_m\}$ is equivalent to the unit vector basis of $l_1$ both in $E_0'$ and in $E_1'$. Applying \eqref{Eq7}, we have
\begin{eqnarray*}
\Big\|\sum_{m=0}^\infty c_mu_m\Big\|_{E_0'}&\ge& \sup\Big\{\int_0^1\Big(\sum_{m=0}^\infty c_mu_m\Big)\Big(\sum_{m=0}^\infty d_mv_m\Big)\,dt:\,\Big\|\sum_{m=0}^\infty d_mv_m\Big\|_{E_0}\le 1\Big\}\\
&\ge&  
\sup\Big\{\sum_{m=0}^\infty c_md_m:\,\qr |d_m|\le B^{-1}\Big\}\ge B^{-1}\sum_{m=0}^\infty |c_m|.
\end{eqnarray*}
Therefore, since $E_1'\subset E_0'$, $\|u_m\|_{E_1'}\le C$ and $\|u_m\|_{E_0'}\le MC$ for all $m=0,1,2,\dots$, where $M$ is the  constant of the embedding $E_0\subset E_1$, we have 
\begin{equation}\label{Eq8}
D^{-1}\sum_{m=0}^\infty |c_m|\leq \Big\|\sum_{m=0}^\infty c_mu_m\Big\|_{E_0'}\le D\sum_{m=0}^\infty |c_m|
\end{equation}
and
\begin{equation}\label{Eq9}
D^{-1}\sum_{m=0}^\infty |c_m|\leq \Big\|\sum_{m=0}^\infty c_mu_m\Big\|_{E_1'}\le D\sum_{m=0}^\infty |c_m|
\end{equation}
for some $D>0$ and all $c_m\in\mathbb{R}$.

{\it Step~3.} Given $1\le q<\infty$, we denote by $F_0$ and $F_1$ the $q$-convexification of the space $E_0'$ and $E_1'$, respectively (if $q=1$ we set $F_0=E_0'$ and $F_1=E_1'$) (see e.g. \cite[1.d, p.~53]{LT-79}). 

Clearly, $F_1\subset F_0$. Further, since the functions $u_m$, $m=0,1,2,\dots$, are pairwise disjoint, then so are the functions $g_m:=|u_m|^{1/q}$, $m=0,1,2,\dots$ Furthermore,
by the definition of the $q$-convexification of a space combined with \eqref{Eq8} and \eqref{Eq9}, for all $c_m\in\mathbb{R}$ we infer 
\begin{equation}\label{Eq10}
D\Big(\sum_{m=0}^\infty |c_m|^q\Big)^{1/q}\leq \Big\|\sum_{m=0}^\infty c_mg_m\Big\|_{F_0}\le D\Big(\sum_{m=0}^\infty |c_m|^q\Big)^{1/q}
\end{equation}
and
\begin{equation}\label{Eq11}
D^{-1}\Big(\sum_{m=0}^\infty |c_m|^q\Big)^{1/q}\leq \Big\|\sum_{m=0}^\infty c_mg_m\Big\|_{F_1}\le D\Big(\sum_{m=0}^\infty |c_m|^q\Big)^{1/q}.
\end{equation}

Moreover, since $F_0=\Lambda_q(\bar{\psi})$ (see Section~\ref{Intro2}), then passing to a subsequence if it is necessary, we may assume that the closed linear span $[g_m]$ is complemented in $F_0$ (see e.g. \cite[Theorem~5.1]{FJT-75}). Let $P$ be a linear projection bounded in $F_0$, whose image coincides with $[g_m]$. Since $F_1\subset F_0$ and the sequence $\{g_m\}$ is equivalent to the unit vector of $l_q$ both in $F_0$ and in $F_1$ (see \eqref{Eq10} and \eqref{Eq11}), $P$ is also bounded in $F_1$, and so the subspace $[g_m]$ is complemented in $F_1$ with the same projection.  

Now, given $1\le p<\infty$, we denote by $Z_{p,q}$ the real Lions-Peetre interpolation space $(F_0,F_1)_{1/2,p}$ (see Section 2).

Observe that the fundamental functions of the spaces $F_0$ and $F_1$ are $\phi_{F_0}(t)=t^{1/(2q)}\log_2^{-1/(2q)}\frac{4}{t}$ and $\phi_{F_1}(t)=t^{1/(2q)}$, respectively.  Also, for an arbitrary symmetric space $X$ we have $\phi_{X'}(t)=t/\phi_X(t)$, $0<t\le 1$ \cite[\S\,II.4]{KPS}.
Thus, applying Formula (1) from \cite[\S\,3.5]{BL} two times together with the duality theorem (see e.g. \cite[Theorem~3.7.1]{BL}), we can identify the fundamental function $\phi_{Z_{p,q}}$ as follows
$$
\phi_{Z_{p,q}}(t)\asymp\phi_{F_0}(t)^{1/2}\phi_{F_1}(t)^{1/2}=t^{1/(2q)}\log_2^{-1/(4q)}\frac{4}{t}.$$
Hence,
\begin{equation*}
\lim_{t\to +0}\frac{\phi_{F_0}(t)}{\phi_{Z_{p,q}}(t)}\asymp\lim_{t\to +0}\log_2^{-1/(4q)}\frac{4}{t}=0,
\end{equation*}
and so, according to \cite[Theorem~4]{A15}, every sequence of the form $\{\frac{\chi_{A_k}}{\phi_{Z_{p,q}}(m(A_k))}\}$, where $A_k$, $k=1,2,\dots$, are pairwise disjoint subsets of $[0,1]$, contains a subsequence equivalent in $Z_{p,q}$ to the unit vector basis of $l_p$. As a result, we conclude that $Z_{p,q}$ is a restricted $p$-DH symmetric space. 

On the other hand, since the above projection $P$ is bounded in $F_0$ and $F_1$, then by inequalities \eqref{Eq10} and \eqref{Eq11} combined together with Baouendi-Goulaouic theorem (see e.g. \cite[Theorem~1.17.1]{Tr}), the sequence $\{g_m\}$ is equivalent in $Z_{p,q}$ to the unit vector basis of $l_q$ and the subspace $[g_m]$ is complemented in $Z_{p,q}$. 

Now, let $1<p,q<\infty$. Recall that $F_0=\Lambda_q(\bar{\psi})$. 
Taking into account that $\gamma_{\bar{\psi}}=1/2$, we conclude that $F_0$ is $q$-convex and $r$-concave for each $r>2q$ (see e.g. \cite[Corollary~2]{Nov-82}). Hence, neither $c_0$ nor $l_1$ is lattice embeddable in $F_0$ and so, by Lozanovsky theorem (see \cite{Loz} or \cite[Theorem~4.71]{AB}), $F_0$ is reflexive. Therefore, the canonical embedding of $F_1$ into $F_0$ is weakly compact and, by Beauzamy theorem \cite{Boz}, the space $Z_{p,q}$ is also reflexive. 

It remains to prove the last assertion of the theorem. 

Let $1\le p<\infty$ and let $\{x_n\}_{n=1}^\infty$ be an arbitrary sequence of pairwise disjoint functions from the space $Z_{p,1}$, $\|x_n\|_{Z_{p,1}}=1$, $n=1,2,\dots$ If $\liminf_{n\to\infty}\|x_n\|_{F_0}=0$, then from \cite[Theorem~4 and Remark~2]{A15} it follows at once that there is a subsequence $\{x_{n_k}\}\subset \{x_n\}$, which spans a complemented subspace in $Z_{p,1}$.

Now, we consider the case when $\liminf_{n\to\infty}\|x_n\|_{F_0}>0$, that is, 
\begin{equation}\label{Eq12}
\|x_n\|_{F_0}\asymp\|x_n\|_{Z_{p,1}}=1,\;\;n=1,2,\dots
\end{equation}
Since $q=1$, we have $F_0=E_0'=\Lambda(\bar{\psi})$. Therefore, applying \cite[Theorem~5.1]{FJT-75} once more (see also \cite{Tok}), we can select a subsequence $\{x_{n_k}\}\subset \{x_n\}$, equivalent in $F_0$ to the unit vector basis of $l_1$, such that the subspace $[x_{n_k}]$ is complemented in $F_0$. Since $Z_{p,1}\subset F_0$, then,  by \eqref{Eq12}, $\{x_{n_k}\}$ is equivalent to the unit vector basis of $l_1$ also in $Z_{p,1}$. Hence, if $Q$ is a bounded projection in $F_0$ with the image $[x_{n_k}]$, for all $x\in Z_{p,1}$ we have
$$
\|Qx\|_{Z_{p,1}}\le C\|Qx\|_{F_0}\le C\|Q\|_{F_0\to F_0}\|x\|_{F_0}\le C'\|Q\|_{F_0\to F_0}\|x\|_{Z_{p,1}},$$ 
that is, $Q$ is bounded in $Z_{p,1}$. Summing up, we conclude that the subspace $[x_{n_k}]$ is complemented in $Z_{p,1}$, and the proof is completed.
\end{proof}

Applying Theorem~\ref{main} in the case when $1\le p\ne q<\infty$, we obtain

\begin{car}\label{solution}
For every $1\le p<\infty$ there exists a restricted $p$-DH symmetric space which is not DH.
\end{car}

\section{$p$-DH property is preserved under isomorphisms}

\begin{Teo}\label{isomorphic}\label{main2}
Let $1\le p<\infty$ and let $X$ be a symmetric space on $[0,1]$, which is isomorphic to a complemented subspace of a $p$-DH symmetric space $Y$. Then either $X=L_2$ (with equivalence of norms) or $X$ is a $p$-DH space.    
\end{Teo}
\begin{proof}
It is well known that every non-separable symmetric space contains a subspace spanned by pairwise disjoint functions, which is isomorphic to $l_\infty$ (see e.g.
\cite[Proposition~1.a.7]{LT-79}). Therefore, since $Y$ is $p$-DH, with $1\le p<\infty$, it is separable. So,
by a deep result by Kalton on uniqueness of rearrangement invariant structures \cite[Theorem~7.4]{K-93} (see also \cite[Theorem~5.7]{Tz}), we can assume that the sequence of Haar functions $\{h_n\}_{n=1}^\infty$ is equivalent in $X$ to some sequence $\{u_n\}_{n=1}^\infty$ of pairwise disjoint functions in $Y$ (let $h_n$ be normalized in $L_\infty$). Let's show  that each normalized block basis of the Haar system contains a subsequence equivalent in $X$ to the unit vector basis $\{e_i\}_{i=1}^\infty$ of $l_p$.

Indeed, let $x_m:=\sum_{k=j_m+1}^{j_{m+1}}a_k^mh_k$, $j_1=0<j_1<j_2<\dots$, $\|x_m\|_X=1$, $m=1,2,\dots$ Then the sequence $\{x_m\}_{m=1}^\infty$ is equivalent in $X$ to the sequence $\{y_m\}_{m=1}^\infty\subset Y$, consisting of pairwise disjoint functions $y_m:=\sum_{k=j_m+1}^{j_{m+1}}a_k^mu_k$, $m=1,2,\dots$ Since $\|y_m\|_Y\asymp 1$, $m=1,2,\dots$, by hypothesis, there is a subsequence $\{y_{m_i}\}$,
which is equivalent in $Y$ to $\{e_i\}$. Therefore, the corresponding subsequence $\{x_{m_i}\}$ is  also equivalent (in $X$) to $\{e_i\}$.
 
Now, suppose that $\{f_{m}\}_{m=1}^\infty$ is an arbitrary  sequence of pairwise disjoint functions in $X$, $\|f_m\|_X=1$, $m=1,2,\dots$ Observe that for each $n=1,2,\dots$ and $m=1,2,\dots$ we have
\begin{equation}\label{int estimate}
\Big|\int_0^1 f_m(t)h_n(t)\,dt\Big|\le
\|f_m\|_X\cdot \|h_n\chi_{{\rm supp}\,f_m}\|_{X'}\le
\phi_{X'}(m({\rm supp}\,f_m)).
\end{equation}
First, suppose that $\phi_{X'}(m({\rm supp}\,f_m))\ge c>0$ for all $m=1,2,\dots$, Then, clearly, $\phi_{X'}(t)\asymp 1$, $0<t\le 1$ (equivalently, $X'=L_\infty$) and so, by \cite[Ch.~II, \S\,4, Formula (4.39)]{KPS}, $\phi_{X}(t)=t/\phi_{X'}(t)\asymp t$, $0<t\le 1$, whence $X=L_1$. Therefore, the space $Y$ is not reflexive and hence, by Lozanovsky theorem (see \cite{Loz} or \cite[Theorem~4.71]{AB}), it contains a subspace spanned by pairwise disjoint functions, which is isomorphic to $c_0$ or $l_1$. Combining this fact with the condition that $Y$ is $p$-DH, $1\le p<\infty$, we conclude that $Y$ is $1$-DH. Since $L_1$ is also $1$-DH, then in this case the theorem is proved.

Let now $\phi_{X'}(m({\rm supp}\,f_m))\to 0$ as $m\to\infty$. Then, from \eqref{int estimate} it follows that  
$$
\lim_{m\to\infty}\int_0^1 f_m(t)h_n(t)\,dt=0$$
 for each $n=1,2,\dots$ It is well known that the Haar system is a basis in every separable symmetric space \cite[Proposition~2.c.1]{LT-79}, and so, in particular, in $X$. Therefore, applying the Bessaga-Pe{\l}czy\'{n}ski Selection Principle \cite[Proposition~1.3.10]{AK}, we can find a subsequence $\{f_{m_i}\}$ equivalent in $X$ to some block basis of the Haar system. Passing once more to a subsequence (and keeping the notation), by the observation from the first part of the proof, we may assume that $\{f_{m_i}\}$ is equivalent in $X$ to the unit vector basis $\{e_i\}$ of $l_p$. Hence, the proof is completed.
\end{proof}

\begin{car}\label{isomorphic1}
(i) If a symmetric space $X$ is isomorphic to a complemented subspace of a 2-DH symmetric space $Y$, then $X$ is also a 2-DH space.

(ii) Let $1\le p<\infty$. If a symmetric space $X$ is isomorphic to a $p$-DH symmetric space $Y$, then $X$ is also a $p$-DH space.
  
\end{car}



\end{document}